%% file: Braun_IndependenceComplexesStableKneserGraphs_Arxiv_Dec2009.tex
\documentclass{amsart}
\usepackage{amssymb,graphicx,color}

\newtheorem{theorem}{Theorem}[section]
\newtheorem{lemma}[theorem]{Lemma}

\theoremstyle{definition}
\newtheorem{definition}[theorem]{Definition}

\theoremstyle{remark}
\newtheorem{remark}[theorem]{Remark}

\newcommand{\Ind}{\mathrm{Ind}}

\begin{document}

\title{Independence Complexes of Stable Kneser Graphs}

\author{Benjamin Braun}
\address{Department of Mathematics,
University of Kentucky, Lexington, KY 40506}
\email{benjamin.braun@uky.edu}

\thanks{The author was partially supported by the NSF through award DMS-0758321.  Thanks to John Shareshian for thoughtful discussions at the beginning of this project.}


\date{\today}


\begin{abstract}
For integers $n\geq 1$, $k\geq 0$, the \emph{stable Kneser graph} $SG_{n,k}$ (also called the \emph{Schrijver graph}) has as vertex set the stable $n$-subsets of $[2n+k]$ and as edges disjoint pairs of $n$-subsets, where a stable $n$-subset is one that does not contain any $2$-subset of the form $\{i,i+1\}$ or $\{1,2n+k\}$.  The stable Kneser graphs have been an interesting object of study since the late 1970's when A. Schrijver determined that they are a vertex critical class of graphs with chromatic number $k+2$.  This article contains a study of the independence complexes of $SG_{n,k}$ for small values of $n$ and $k$.  Our contributions are two-fold: first, we find that the homotopy type of the independence complex of $SG_{2,k}$ is a wedge of spheres of dimension two.  Second, we determine the homotopy types of the independence complexes of certain graphs related to $SG_{n,2}$.
\end{abstract}

\maketitle

% Body of text  %%%%%%%%%%%%%%%%%%%%%%%%%%%%%%%%%%%%%%%

\section{Introduction}

Let $[n]:=\{1,2,3,\ldots,n\}$ and consider the following family of graphs.

\begin{definition}  For each pair of integers $n\geq 1$, $k\geq 0$, the \textit{Kneser graph} $KG_{n,k}$ has as vertices the $n$-subsets of $[2n+k]$ with edges defined by disjoint pairs of $n$-subsets.  For the same parameters, the \textit{stable Kneser graph} $SG_{n,k}$, also called the \emph{Schrijver graph}, is the subgraph of $KG_{n,k}$ induced by the \textit{stable} $n$-subsets of $[2n+k]$, i.e. those $n$-subsets that do not contain any $2$-subset of the form $\{i,i+1\}$ or $\{1,2n+k\}$.
\end{definition}

The Kneser and stable Kneser graphs have interesting properties related to independent sets of vertices, where a collection of vertices in a graph $G$ is an \textit{independent set} if the vertices are pairwise non-adjacent in $G$.  Perhaps the most widely-known structure in graph theory related to independent sets is that of a \emph{proper graph coloring}, i.e. a partition of the vertices of $G$ into disjoint independent sets.  The minimal number of independent sets required for such a partition is called the \emph{chromatic number} of $G$ and is denoted $\chi(G)$.  In 1978, L.~Lov\'{a}sz proved in \cite{LovaszChromaticNumberHomotopy} that $\chi(KG_{n,k})=k+2$ by using an ingenious application of the Borsuk-Ulam theorem, thus verifying a conjecture due to M.~Kneser from $1955$.  Shortly afterwards, A.~Schrijver determined in \cite{SchrijverGraphs} that $\chi(SG_{n,k})=\chi(KG_{n,k})$, again using the Borsuk-Ulam theorem.  Schrijver also proved that the stable Kneser graphs are vertex critical, i.e. the chromatic number of any subgraph of a stable Kneser graph $SG_{n,k}$ obtained by removing vertices is strictly less than $\chi(SG_{n,k})$.  These theorems were one source of inspiration for subsequent work involving the interaction of combinatorics and algebraic topology, see \cite{JonssonBook,KozlovBook,MatousekBorsukUlam} for recent textbook accounts of further developments.

Recall that an (abstract) simplicial complex $\Delta=(V,\mathcal{F})$ is a finite set $V$, called the \emph{vertices} of $\Delta$, together with a collection of subsets $\mathcal{F}\subseteq 2^{V}$ such that \[F\in \mathcal{F},G\subseteq F \Rightarrow G\in \mathcal{F},\] called the \emph{faces} of $\Delta$.  For technical purposes, we include the emptyset as a face.  Lov\'{a}sz's original proof that $\chi(KG_{n,k})=k+2$ followed from a general theorem bounding $\chi(G)$ from below by a function of the connectivity of the \emph{neighborhood complex} of $G$, the complex whose vertices are the vertices of $G$ and whose faces are vertices sharing a common neighbor.  In \cite{BjornerDeLongueville}, Bj\"{o}rner and De Longueville proved that the neighborhood complex of $SG_{n,k}$ is homotopy equivalent to a $k$-sphere, implying that $SG_{n,k}$ is well-behaved topologically with regard to this construction and Lov\'{a}sz's theorem.  While the neighborhood complex plays a fundamental role in providing topological lower bounds on chromatic numbers, this is not the only topological construction that investigates independence structures.  If one is interested in the interplay among all the independent sets in $G$, without regard to chromatic numbers, one is led to the following construction.  

\begin{definition} Let $G=(V,E)$ be a graph.  The \textit{independence complex} of $G$, $\Ind(G)$, is the simplicial complex with vertex set $V$ and faces given by independent sets.
\end{definition}

Independence complexes have been the subject of recent investigation, see for example \cite{LinussonGridGraphs,EhrenborgHetyei,EngstromClawFree,EngstromWitten,JonssonHardSquares,Thapper}.  Five of these papers involve a connection between independence complexes of graphs and hard squares models in statistical mechanics.  Additionally, the homotopy type of the independence complexes of cycles played a critical role in the recent resolution by E.~Babson and D.~Kozlov in \cite{BabsonKozlovLovaszConjecture} of Lov\'{a}sz's conjecture regarding odd cycles and graph homomorphism complexes. 

Our purpose in this paper is to investigate the homotopy type of the independence complexes of the stable Kneser graphs $SG_{2,k}$ and the independence complexes of a family of graphs related to $SG_{n,2}$.  There are several reasons to be curious about the independence complex of $SG_{n,k}$.  Since the stable Kneser graphs are vertex-critical, they are a minimal obstruction to colorability in the Kneser graphs, and the chromatic number inherently measures some restricted behavior of independent sets which is reflected in the neighborhood complex.  It is of interest to see what, if any, properties of independent sets in $SG_{n,k}$ are exposed through $\Ind(SG_{n,k})$.  Also, an independent set in $SG_{n,k}$ is a pairwise intersecting family of stable $n$-subsets of $[2n+k]$.  Such families have been previously studied from an extremal perspective, see for example \cite{Talbot} and the references therein.

The homotopy types of the independence complexes of some stable Kneser graphs are already known.  For $n=1$, the stable Kneser graphs are complete graphs and thus their independence complexes are wedges of $0$-dimensional spheres.  For $k=0$, the stable Kneser graphs are complete graphs on two vertices, hence their independence complexes are zero dimensional spheres.  For $k=1$, it is easy to see that $SG_{n,1}=C_{2n+1}$, the cycle of odd length, and the homotopy types of independence complexes of cycles are known.

\begin{theorem}\label{cycle}{\rm(Kozlov, \cite{KozlovIndCycles})} For $n\geq 1$, let $C_n$ denote the cycle of length $n$.  The following homotopy equivalence holds:
\[\Ind(C_n)\simeq \left\{ \begin{array}{ll}
S^{r-1}\bigvee S^{r-1} & if\phantom{1} n=3r \\
S^{r-1} & if\phantom{1} n=3r\pm 1
\end{array} \right.
\]
\end{theorem}

Our first contribution is to describe $\Ind(SG_{2,k})$ up to homotopy.

\begin{theorem}\label{n=2}
For $k\geq 4$, \[ \Ind(SG_{2,k})\simeq \bigvee_{\substack{\frac{(k-3)(k-1)(k+4)}{6}-1}}S^2. \]
Also,
\[ \Ind(SG_{2,2})\simeq S^1 \bigvee S^1\]
and \[\Ind(SG_{2,3})\simeq S^1.\]
\end{theorem}

Our second contribution is to investigate $\Ind(SG_{n,2})$.  Unfortunately, as will be discussed in Section~\ref{SGn2}, these complexes are more complicated than those where $n=2$, and their homotopy type is still unknown.  However, we will be able to determine the homotopy type of a class of graphs we call $E_{2n+2}$, close relatives of $SG_{n,2}$ that will be fully defined in Section~\ref{SGn2}.  A rough description of $SG_{n,2}$ is as a cylinder graph with some additional edges on the cycles forming the ends of the cylinder; $E_{2n+2}$ is then obtained by ``squeezing'' $SG_{n,2}$ so that only the end cycles and additional edges remain.

\begin{theorem}\label{E2n} Let $n\geq 3$.  If $4 \mid 2n+2$, i.e. $n$ is odd, then
\[ \Ind(E_{2n+2}) \simeq \left\{ \begin{array}{ll}
S^{2k+1} \bigvee S^{2k+1} \bigvee S^{2k+1} & if\phantom{1} n=4k+1 \\
S^{2k+2} & if\phantom{1} n=4k+3
\end{array} \right. .
\]
If $4 \nmid 2n+2$, i.e. $n$ is even, then
\[ \Ind(E_{2n+2}) \simeq \left\{ \begin{array}{ll}
S^{2k} & if\phantom{1} n=6k \\
S^{2k+1}\bigvee S^{2k+1} & if\phantom{1} n=6k+2 \\
S^{2k+2} & if\phantom{1} n=6k+4
\end{array} \right. .
\]

\end{theorem}

It is interesting to compare the case in Theorem~\ref{E2n} where $n$ is even with Theorem~\ref{cycle}, as we will later see that when $n$ is even the graph $E_{2n+2}$ is closely related to an odd cycle of length $n+1$.

Our primary tool for proving these theorems is discrete Morse theory.  The remainder of the paper is structured as follows.  In Section~\ref{DMT}, we discuss the basics of discrete Morse theory, including the construction of acyclic partial matchings via \emph{matching trees} introduced in \cite{LinussonGridGraphs}.  In Section~\ref{proofn=2}, we prove Theorem~\ref{n=2}.  In Section~\ref{SGn2}, we provide an explicit description of the graphs $SG_{n,2}$ leading to the definition of $E_{2n+2}$ and remark on a connection between $\Ind(SG_{n,2})$ and \cite{JonssonHardSquares}.  Finally, in Section~\ref{proofE2n} we provide a proof of Theorem~\ref{E2n}.

%%%%%%%%%%%%%%%%%%%%%%%%%%%%%%%%%%%%%%%%%%%%%%%%%%%%%%%%

\section{Tools From Discrete Morse Theory}\label{DMT}

In this section we introduce the tools we need from discrete Morse theory.  Discrete Morse theory was introduced by R.~Forman in \cite{FormanMorseTheory} and has since become a standard tool in topological combinatorics.  The main idea of (simplicial) discrete Morse theory is to pair cells in a simplicial complex in a manner that allows them to be cancelled via elementary collapses, reducing the complex under consideration to a homotopy equivalent complex, cellular but possibly non-simplicial, with fewer cells.  Detailed discussions of the following definitions and theorems, along with their proofs, can be found in \cite{JonssonBook,KozlovBook}.

\begin{definition}\label{Morsedef}
A \textit{partial matching} in a poset $P$ is a partial matching in the underlying graph of the Hasse diagram of $P$, i.e. it is a subset $M\subseteq P\times P$ such that \begin{itemize}
\item $(a,b)\in M$ implies $b$ covers $a$, i.e. $a<b$ and no $c$ satisfies $a<c<b$, and
\item each $a \in P$ belongs to at most one element in $M$.
\end{itemize}
When $(a,b)\in M$ we write $a=d(b)$ and $b=u(a)$.  A partial matching on $P$ is called \textit{acyclic} if there does not exist a cycle \[b_1>d(b_1)<b_2>d(b_2)<\cdots<b_n>d(b_n)<b_1,\] with $n\geq 2$, and all $b_i\in P$ being distinct.
\end{definition}

Given an acyclic partial matching $M$ on $P$, we say that the unmatched elements of $P$ are \textit{critical}.  The following theorem asserts that an acyclic partial matching on the face poset of a polyhedral cell complex is exactly the pairing needed to produce our desired homotopy equivalence.

\begin{theorem}\label{mainmorse} {\rm (Main Theorem of Discrete Morse Theory)}
Let $\Delta$ be a polyhedral cell complex and let $M$ be an acyclic partial matching on the face poset of $\Delta$.  Let $c_i$ denote the number of critical $i$-dimensional cells of $\Delta$.  The space $\Delta$ is homotopy equivalent to a cell complex $\Delta_c$ with $c_i$ cells of dimension $i$ for each $i\geq 0$, plus a single $0$-dimensional cell in the case where the emptyset is paired in the matching.
\end{theorem}

\begin{remark}\label{sphereremark}In particular, if an acyclic partial matching $M$ has critical cells only in a fixed dimension $i$, then $\Delta$ is homotopy equivalent to a wedge of $i$-dimensional spheres.
\end{remark}

It is often useful to be able to make acyclic partial matchings on different sections of a poset and combine them to form a larger acyclic partial matching.  This process is formalized via the following theorem, referred to as the \emph{Cluster Lemma} in \cite[Lemma~4.2]{JonssonBook} and the \emph{Patchwork Theorem} in \cite[Theorem~11.10]{KozlovBook}.

\begin{theorem}\label{patchwork}
If $\phi:P\rightarrow Q$ is an order-preserving map and, for each $q \in Q$, each subposet $\phi^{-1}(q)$ carries an acyclic partial matching $M_q$, then the union of the $M_q$ is an acyclic partial matching on $P$.
\end{theorem}

To facilitate the study of $\Ind(G)$ for a graph $G=(V,E)$, \emph{matching trees} were introduced by Bousquet-M\'{e}lou, et al. in \cite[Section~2]{LinussonGridGraphs}.  Let \[\Sigma(A,B):=\left\{ I\in \Ind(G) : A\subseteq I \phantom{.} \mathrm{and} \phantom{.} B\cap I = \emptyset \right\},\] where $A,B\subseteq V$ satisfy $A\cap B = \emptyset$ and $N(A):= \cup_{a\in A}N(a) \subseteq B$, where $N(a)$ denotes the neighbors of $a$ in $G$.

\begin{definition}
Let $G$ be a connected graph.  A \emph{matching tree}, $M(G)$, for $G$ is a directed tree constructed according to the following algorithm.  Begin by letting $M(G)$ be a single node labeled $\Sigma(\emptyset,\emptyset)$, and consider this node a sink until after the first iteration of the following loop:

\phantom{.}

\noindent \textbf{WHILE} $M(G)$ has a leaf node $\Sigma(A,B)$ that is a sink with $|\Sigma(A,B)|\geq 2$

\noindent \textbf{DO ONE OF THE FOLLOWING}
\begin{enumerate}
\item If there exists a vertex $p\in V\setminus (A\cup B)$ such that \[|N(p)\setminus (A\cup B)|=0,\] create a directed edge from $\Sigma(A,B)$ to a new node labeled $\emptyset$.  Refer to $p$ as a \emph{free vertex} of $M(G)$.
\item If there exist vertices $p\in V\setminus (A\cup B)$ and $v\in N(p)$ such that \[N(p)\setminus (A\cup B)=\{v\},\] create a directed edge from $\Sigma(A,B)$ to a new node labeled \[\Sigma(A\cup \{v\},B\cup N(v)).\]  Refer to $v$ as a \emph{matching vertex} of $M(G)$ with respect to $p$.
\item Choose a vertex $v\in V\setminus (A\cup B)$ and created two directed edges from $\Sigma(A,B)$ to new nodes labeled
\[\Sigma(A,B\cup \{v\})\]
and
\[\Sigma(A\cup \{v\},B\cup N(v)).\]  Refer to $v$ as a \emph{splitting vertex} of $M(G)$.
\end{enumerate}

\phantom{.}

The node $\Sigma(\emptyset,\emptyset)$ is called the \emph{root} of the matching tree, while any non-root node of degree $2$ in $M(G)$ is called a \emph{matching site} of $M(G)$ and any non-root node of degree $3$ is called a \emph{splitting site} of $M(G)$.
\end{definition}

The key observation in \cite{LinussonGridGraphs} is that a matching tree on $G$ yields an acyclic partial matching on the face poset of $\Ind(G)$, as the following theorem indicates.

\begin{theorem}{\rm \cite[Section~2]{LinussonGridGraphs} } A matching tree $M(G)$ for $G$ yields an acyclic partial matching on the face poset of $\Ind(G)$ whose critical cells are given by the non-empty sets $\Sigma(A,B)$ labeling non-root leaves of $M(G)$.  In particular, for such a set $\Sigma(A,B)$, the set $A$ yields a critical cell in $\Ind(G)$.
\end{theorem}

%%%%%%%%%%%%%%%%%%%%%%%%%%%%%%%%%%%%%%%%%%%

\section{Proof of Theorem~\ref{n=2}}\label{proofn=2}

The case $k=2$ is a simple exercise that can be completed by drawing $\Ind(SG_{2,2})$ on a sheet of paper (the complex is pure two-dimensional and has only eight maximal faces).  The case $k=3$ is handled at the end of this section.  Thus, we need to first consider the situation where $k \geq 4$.  Let $Q_k$ denote the face poset of $\Ind(SG_{2,k})$ and let $I_{k+2}$ be a $(k+2)$-element chain, with elements labeled $3<4<5<6<\cdots <k+4$.  Our goal is to create an acyclic partial matching on $Q_k$ by using Theorem~\ref{patchwork} to break $Q_k$ into preimages and produce acyclic partial matchings on these.  Note that we consider below only sets $\{i,j\}$ that are stable, regardless of the parameters provided.  In this section, addition and subtraction are modulo $k+4$.

Observe that the maximal elements of $Q_k$ are of two types:
\begin{itemize}
\item[\textbf{W:}] A \emph{wheel through $i$} is an independent set of the form \[W_i:=\{\{i,j\}:j\in [k+4]\}.\]
\item[\textbf{T:}] A \emph{triangle} is an independent set of the form \[T_{i,j,h}:=\{\{\{i,j\},\{j,h\},\{i,h\}\}:i,j,h\in [k+4]\}.\]
\end{itemize}
Thus, $\Ind(SG_{2,k})$ is built from $k+4$ simplices of dimension $k$ corresponding to wheels, with additional $2$-cells corresponding to triangles.  

Let $\phi:Q_k\rightarrow I_{k+2}$ be defined as follows.  Note that all variables referenced (e.g. $j,r,i_1,i_2,\ldots$) are integers in $[k+4]$.
\[
\phi^{-1}(3)  := \left\{ \begin{array}{ll}
\emptyset & \\
\sigma & \sigma \subseteq W_1 \\ 
\sigma & \sigma \subseteq W_3 \\
\{\{1,j\},\{3,j\}\} & 3<j<k+4 \\
T_{1,3,j} & 3<j<k+4 \\
\end{array}\right.
\]
For $3<l<k+4$,
\[
\phi^{-1}(l)  := \left\{ \begin{array}{ll}
\{2,l\} & \\
\{\{1,l\},\{2,l\}\} & \\
\{l,j\} & l<j\leq k+4 \\
\{\{l,i\},\{l,j\}\} & l<i<j\leq k+4 \\
\{\{i,l\},\{l,j\}\} & 1\leq i<l<j \leq k+4 \\
\{\{i,l\},\{j,l\}\} & 2\leq i<j<l \\
\{\{l,i_1\},\{l,i_2\},\ldots,\{l,i_r\}\} & 3\leq r\leq k+1 \\
\{\{1,j\},\{l,j\}\} & l<j\leq k+3 \\
T_{1,l,j} & l<j\leq k+3 \\
\end{array}\right.
\]
\[\phi^{-1}(k+4)  := \left\{ \begin{array}{ll}
\{2,k+4\} & \\
\{\{2,i_1\},\{2,i_2\},\ldots,\{2,i_r\}\} & 2\leq r\leq k+1 \\
\{\{k+4,i_1\},\{k+4,i_2\},\ldots,\{k+4,i_r\}\} & 2\leq r\leq k+1 \\
T_{i,j,h} & i,j,h\neq 1 \\
\end{array}\right.
\]

It is straightforward to check that this is an order-preserving map.  For $3\leq l < k+4$, we produce an acyclic matching $M_l$ on $\phi^{-1}(l)$ via the matching $(\sigma, \sigma \cup \{1,l\})$ for each $\sigma$ not containing $\{1,l\}$.  It is straightforward to check that for each $l$, every element of $\phi^{-1}(l)$ is an element of some matched pair in $M_l$.  $M_l$ is acyclic because if one attempts to construct a directed cycle in $\phi^{-1}(l)$ as in Definition~\ref{Morsedef}, starting from an element $\sigma$ not containing $\{1,l\}$, we must begin our cycle with \[\sigma < \sigma \cup \{1,l\} < (\sigma \cup \{1,l\}) \setminus \{i,j\} \] for some $\{i,j\}\neq \{1,l\}$.  However, there is no $\tau \in Q_k$ such that $\sigma':=(\sigma \cup \{1,l\}) \setminus \{i,j\}$ satisfies $(\sigma',\tau)\in M_l$, hence we cannot complete our desired cycle.

Remaining is only the poset $\phi^{-1}(k+4)$.  To establish an acyclic partial matching here, we apply Theorem~\ref{patchwork} a second time.  Let $C:=\{b<r_6<r_7<\cdots<r_{k+3}<m_1<m_2<t_2<s_4<s_5<\cdots<s_{k+2}<m_3<m_4<t_{k+4}\}$ be a chain, and define $\psi:\phi^{-1}(k+4)\rightarrow C$ as follows.

\[
\psi^{-1}(b) := \left\{  \begin{array}{l}
\{2,k+4\} \\
\{\{2,4\},\{2,k+4\}\} \\
\end{array}\right.
\]

For $6\leq i \leq k+3$,
\[
\psi^{-1}(r_i)  := \left\{  \begin{array}{l}
\{\{2,4\},\{2,i\}\} \\
\{\{2,4\},\{2,i\},\{4,i\}\} \\
\end{array}\right.
\]

\[\psi^{-1}(m_1):= \left\{ \begin{array}{l}
\{\{2,5\},\{2,7\}\} \\
\{\{2,5\},\{2,7\},\{5,7\}\} \\
\end{array}\right.
\]

\[\psi^{-1}(m_2):= \left\{ \begin{array}{l}
\{\{2,4\},\{2,5\}\} \\
\{\{2,4\},\{2,5\},\{2,7\}\} \\
\end{array}\right.
\]

\[
\psi^{-1}(t_2)  := \left\{ \begin{array}{ll}
\{\{2,i\},\{2,j\}\} & 5\leq i<j\leq k+4 \\
\text{ except } \{\{2,5\},\{2,7\}\} &  \\
\{\{2,i_1\},\{2,i_2\},\ldots,\{2,i_r\}\} & 3\leq r\leq k+1 \\
\text{ except } \{\{2,4\},\{2,5\},\{2,7\}\} & \\
\end{array}\right.
\]

For $4\leq j\leq k+2$,
\[
\psi^{-1}(s_j)  := \left\{ \begin{array}{l}
\{\{2,k+4\},\{j,k+4\}\} \\
\{\{2,k+4\},\{j,k+4\},\{2,j\}\} \\
\end{array}\right.
\]

\[\psi^{-1}(m_3):= \left\{ \begin{array}{l}
\{\{3,k+4\},\{5,k+4\}\} \\
\{\{3,k+4\},\{5,k+4\},\{3,5\}\} \\
\end{array}\right.
\]

\[\psi^{-1}(m_4):= \left\{ \begin{array}{l}
\{\{2,k+4\},\{3,k+4\}\} \\
\{\{2,k+4\},\{3,k+4\},\{5,k+4\}\} \\
\end{array}\right.
\]

\[
\psi^{-1}(t_{k+4})  := \left\{ \begin{array}{ll}
\{\{i,k+4\},\{j,k+4\}\} & 3\leq i < j < k+4 \\
\text{ except } \{\{3,k+4\},\{5,k+4\}\} &  \\
\{\{k+4,i_1\},\{k+4,i_2\},\ldots,\{k+4,i_r\}\} & 3\leq r\leq k+1 \\
\text{ except } \{\{2,k+4\},\{3,k+4\},\{5,k+4\}\} & \\
T_{i,j,k} & T_{i,j,k} \text{ not yet listed} \\
\end{array}\right.
\]

It is straightforward to check that this is an order preserving map.  To form acyclic partial matchings on these preimages, for all elements in the chain except $t_2$ and $t_{k+4}$, match the pair of elements in the preimage.  Match the pair $(\sigma,\sigma \cup \{2,4\})$ on $\psi^{-1}(t_2)$; this matching is acyclic and matches every element.  Match the pair $(\sigma,\sigma \cup \{2,k+4\})$ on $\psi^{-1}(t_{k+4})$; this matching is again acyclic, but does not match every element.  One can check that the critical cells on $\psi^{-1}(t_{k+4})$ are given by the set 
\[M_{\mathrm{crit}}:=\{T_{i,j,h}:i<j<h, \phantom{.} i,j,h\neq 1\}\setminus S,\]
where 
\[S:= \{T_{2,4,j}:6\leq j \leq k+3\} \cup \{T_{2,j,k+4}:4\leq j \leq k+2\} \cup \{T_{3,5,k+4},T_{2,5,7}\}.\]  
It is straightforward to calculate that the size of $\{T_{i,j,h}:i<j<h, \phantom{.} i,j,h\neq 1\}$ is $\binom{k+1}{3}$ and the size of $S$ is $2k-1$, hence the size of $M_{\mathrm{crit}}$ is \[\binom{k+1}{3}-(2k-1) = \frac{(k-3)(k-1)(k+4)}{6}-1,\] as desired.  We are now in a position to invoke Theorems~\ref{mainmorse}~and~\ref{patchwork}, completing our proof for the case $k\geq 4$.

For the case $k=3$, we define $\phi$ as above, but we must modify the definition of $\psi$.  In particular, for this case we eliminate $\psi^{-1}(m_1)$ and $\psi^{-1}(m_2)$ and include $\{\{2,5\},\{2,7\}\},\{\{2,4\},\{2,5\}\},\text{ and }\{\{2,4\},\{2,5\},\{2,7\}\}$ in $\psi^{-1}(t_{2})$.  Note that the triangle $T_{2,5,7}$ is already present in $\psi^{-1}(s_5)$.  There are only four stable triangles avoiding $1$, namely $T_{2,4,6},T_{2,4,7},T_{2,5,7},\text{ and }T_{3,5,7}$.  These four cells are paired individually in the preimages of $r_6,s_4,s_5,\text{ and }m_3$, respectively, implying that there are no unlisted triangles contained in $\psi^{-1}(t_7)$.  Thus, for the case $k=3$, the only critical cell is $\{\{2,4\},\{2,5\}\}$, and our proof is complete.

%%%%%%%%%%%%%%%%%%%%%%%%%%%%%%%%%%%%%%%%%%%%%%%%%%%%%%%%%%%%%%%%

\section{The graphs $SG_{n,2}$ and $E_{2n+2}$}\label{SGn2}

The graphs $SG_{n,2}$ admit an alternate description which we will discuss here.  This description is given in detail in \cite{BraunSymmetry}, but the details are easy to fill in from the following.  For an integer $n\geq 2$, we define:
\begin{itemize}
\item[] $p(n):= \left\{ \begin{array}{ll}
n & \mathrm{if}\phantom{.} 2\nmid n \\
n-1 & \mathrm{otherwise}
\end{array}\right.$
\item[] $o(n):= \left\{ \begin{array}{ll}
\frac{n+1}{2} & \mathrm{if} \phantom{.} 2\mid n+1 \\
\frac{n+2}{2} & \mathrm{otherwise}
\end{array}\right.$
\end{itemize}
Observe that a vertex of $SG_{n,2}$ is given by a stable $n$-subset of $[2n+2]$ and that these subsets may be partitioned into three classes.  

\begin{definition} Let $X:=\{i_1,i_2,\ldots, i_n\}$ be a stable $n$-subset of $[2n+2]$, where $[2n+2]$ is ordered cyclically.
\begin{itemize}
\item[A:] We say $X$ is an \emph{alternating end vertex} if $X$ is an image \[\alpha(\{1,3,\ldots,p(n),p(n)+3,p(n)+5,\ldots,2n\})\] for some permutation $\alpha$ of the stable $n$-subsets of $[2n+2]$ induced by a cyclic permutation of $[2n+2]$;
\item[B:] We say $X$ is a \emph{bipartite end vertex} if $X$ is a proper subset of either the even numbers or the odd numbers; and
\item[M:] We say $X$ is a \emph{middle vertex} for all other cases. 
\end{itemize}
\end{definition}

\begin{figure}[ht]
\begin{center}
\includegraphics{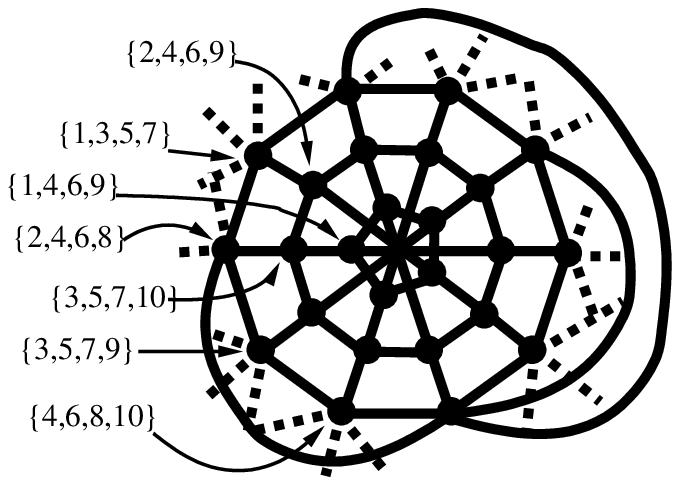}
\end{center}
\caption{$SG_{4,2}$}
\label{SGFig1}
\end{figure}

\begin{figure}[ht]
\begin{center}
\includegraphics{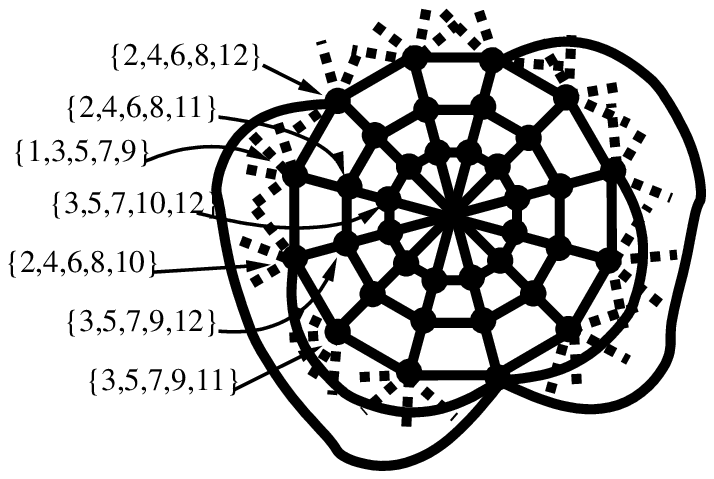}
\end{center}
\caption{$SG_{5,2}$}
\label{SGFig2}
\end{figure}

Figures~\ref{SGFig1} and~\ref{SGFig2} contain representations of $SG_{4,2}$ and $SG_{5,2}$ which we will use as references throughout this discussion. Recall that given two graphs $G$ and $H$, the \emph{cartesian product} $G\Box H$ has vertex set $V(G)\times V(H)$ with $(u,v)$ adjacent to $(u',v')$ if and only if $u=u'$ and $\{v,v'\}\in E(H)$, or $v=v'$ and $\{u, u'\}\in E(G)$. Note that each copy of $SG_{n,2}$ contains a complete bipartite graph $K_{n+1,n+1}$ induced by the bipartite end vertices of $SG_{n,2}$, hence the nomenclature.  In Figures~\ref{SGFig1} and~\ref{SGFig2}, these vertices are the outer ring of the graph, where for visual clarity we have not displayed all the edges, replacing them instead with dashed lines at each vertex indicating the presence of an edge.  Let $C_{j}$ denote the cycle of length $j$ and let $P_{j}$ denote the path with $j$ vertices.  For every $n$, the middle vertices of $SG_{n,2}$ induce as a subgraph a copy of $C_{2n+2}\Box P_{o(n)-2}$ which we call the middle cylinder.  In our examples above, $o(4)=o(5)=3$, and hence the middle rings of vertices in Figures~\ref{SGFig1} and~\ref{SGFig2} are $C_{10}\Box P_{1}$ and $C_{12}\Box P_1$, respectively.  It is easy to see that each bipartite end vertex of $SG_{n,2}$ is connected to a unique vertex on an end cycle of the middle cylinder.  Finally, the alternating end vertices of $SG_{n,2}$ induce a copy of $C_{n+1}$ in the case when $2\mid n$ and a copy of $DC_{2n+2}$ in the case when $2\nmid n$, where $DC_{2n+2}$ is defined to be a $(2n+2)$-cycle augmented by edges connecting antipodal vertices.  In the first case, each alternating end vertex of $SG_{n,2}$ is connected to a pair of middle vertices, while in the latter case each alternating end vertex of $SG_{n,2}$ is connected to a unique middle vertex.  A moment of thought reveals that when $n$ is even, there are $n+1$ alternating end vertices, while for $n$ odd there are $2n+2$ alternating end vertices, and there are $2n+2$ bipartite end vertices and $(2n+2)\cdot (o(n)-2)$ middle vertices.

\begin{remark} In the case where $n$ is odd, the graph $SG_{n,2}$ can be described as an even cylinder $C_{2n+2}\Box P_{o(n)}$ with additional edges forming a complete bipartite graph on one end and a cycle with diagonals on the other end, with a similar description for $n$ even.  Though the graphs $SG_{n,2}$ admit this nice description, the complexes $\Ind(SG_{n,2})$ have been resistant to investigation.  The difficult nature of this problem is similar to the difficult nature of the study of $\Ind(C_{2n+2}\Box P_l)$ for arbitrary $l\geq 6$, $n\geq 3$.  As discussed in Section~6 of \cite{JonssonHardSquares}, even a determination of the Euler characteristic of $\Ind(C_{2n+2}\Box P_l)$ is unknown in general.  For small values of $n$ and $l$, there are some results regarding the homotopy type and Euler characteristic of these complexes, e.g. \cite{Thapper}, but in general this is an interesting open problem.
\end{remark}

One approach for studying $\Ind(SG_{n,2})$ would be to consider arbitrary even circumference cylinders $C_{2n+2}\Box P_l$ and augment their end cycles in a manner consistent with that described above.  One might hope to then induct on $l$ in some fashion, though our attempts following this approach have been unsuccessful.  However, were one to pursue this strategy, the base case would be a cylinder of length two, i.e. a cylinder with no middle vertices.  We therefore introduce the following class of graphs, which are formed by considering only the end vertices of $SG_{n,2}$ and connecting them directly via edges.  Let $K_{n+1,n+1}$ denote the complete bipartite graph on vertex set $\{1,2,\ldots,2n+2\}$ with bipartition into the set of odds and evens.  Let $DC_{2n+2}$ denote the cycle on the vertex set $\{c_1,c_2,\ldots,c_{2n+2}\}$ with edges $\{\{c_i,c_{i+1}\}: i\in [2n+2]\}\cup\{\{c_i,c_{i+n+1}\}:i\in [n+1]\}$.  Let $C_{n+1}$ denote the cycle on the vertex set $\{c_1,c_3,c_5,\ldots,c_{2n+1}\}$ with edges $\{\{c_i,c_{i+2}\}:i\in \{1,3,5,\ldots,2n+1\}\}$.  For all of these graphs, addition of indices for vertices is modulo $2n+2$.  

\begin{definition}
For $n\geq 2$, let $E_{2n+2}$ denote the following graph: if $2\nmid n$, take a copy of $K_{n+1,n+1}$ and a copy of $DC_{2n+2}$ and add an edge connecting each vertex $i$ of $K_{n+1,n+1}$ to the vertex $c_i$ of $DC_{2n+2}$.  If $2\mid n$, take a copy of $K_{n+1,n+1}$ and a copy of $C_{n+1}$ and add edges connecting each vertex $c_i$ of $C_{n+1}$ to the vertices $i$ and $i+n+1$ of $K_{n+1,n+1}$.
\end{definition}

As Theorem~\ref{E2n} indicates, the topology of $\Ind(E_{2n+2})$ is reasonably well-behaved.  It would be of interest to understand more about the topology of the independence complexes of the graphs obtained by adding middle cycles back into $E_{2n+2}$; in particular, it would be interesting to know if there is any relation between the independence complexes of the graphs obtained by adding $j$ middle cycles and $j+k$ middle cycles to $E_{2n+2}$ for small values of $k$.

%%%%%%%%%%%%%%%%%%%%%%%%%%%%%%%%%%%%%%%%%%%%%%%%%%%%%%%%%%%%%%%%%

\section{Proof of Theorem~\ref{E2n}}\label{proofE2n}

To prove Theorem~\ref{E2n}, we will construct matching trees for the graphs $E_{2n+2}$, dividing our proof into the cases where $4\mid 2n+2$ and $4\nmid 2n+2$.  We will refer to vertices of the matching tree as \emph{nodes}, reserving the word \emph{vertices} for the vertices of $E_{2n+2}$.  During the construction of our trees, references to matching, splitting, and free vertices of $E_{2n+2}$ are references to the types of vertices possible in the matching tree construction algorithm given in Section~\ref{DMT}.  Throughout, it might be helpful for the reader to use diagrams like Figures~\ref{ELPicProof},~\ref{nOddPic}, and \ref{nEvenPic} to illustrate our matching trees, which we typically describe with $\Sigma(A,B)$ notation; in these diagrams, each node of the matching tree is shown as a graph representing $\Sigma(A,B)$, where the black dots are elements of $A$, the white dots are elements of $B$, and the gray dots are in neither $A$ nor $B$.  Before we begin, we provide a definition and three lemmas.

\begin{definition} Let $EL_0:=K_2$ and let $EL_1:=K_{1,3}$, where $K_{1,3}$ is a complete bipartite graph with bipartition sets of size $1$ and $3$.  For $r\geq 2$, let $EL_r$ be the graph with $2r+2$ vertices depicted in Figure~\ref{ELPic}.
\begin{figure}[ht]
\begin{center}
\input{ELPic.pstex_t}
\end{center}
\caption{}
\label{ELPic}
\end{figure}
\end{definition}

\begin{lemma}\label{ELlemma} There is a matching tree on $EL_r$ with a single non-empty leaf corresponding to a critical cell of size
\[
\left\{
\begin{array}{ll}
2k+1 & \text{ if  } r=4k \\
2k+1 & \text{ if  } r=4k+1 \\
\end{array}
\right. ,
\]
and no critical cell if $r=4k+2$ or $r=4k+3$.
\end{lemma}

\begin{proof}
Figure~\ref{ELPicProof} demonstrates a matching tree taking $EL_r$ to $EL_{r-4}$ via matching vertices.  Iterate this process until the only gray vertices remaining form a copy of $EL_r$ for $r=0,1,2,\text{ or }3$.  At each iteration, exactly two black vertices are added.  It is now an easy exercise to check that there are matching trees on $EL_0$ and $EL_1$ yielding a single non-empty leaf with a single black vertex, while there are matching trees on $EL_2$ and $EL_3$ with no critical cells.
\begin{figure}[ht]
\begin{center}
\input{ELPicProof.pstex_t}
\end{center}
\caption{}
\label{ELPicProof}
\end{figure}
\end{proof}

\begin{lemma}\label{pathlemma}{\rm(\cite[Prop 11.16]{KozlovBook})} For a path $P_n$ with $n$ vertices and $n-1$ edges, there is a matching tree $M(P_n)$ with a single non-empty leaf corresponding to a critical cell of size
\[
\left\{
\begin{array}{ll}
k & \text{ if  } n=3k \\
k+1 & \text{ if  } n=3k+2 \\
\end{array}
\right. ,
\]
and no critical cells if $n=3k+1$.
\end{lemma}

\begin{lemma}\label{cyclelemma}{\rm(\cite[Prop 11.17]{KozlovBook})} For a cycle $C_n$ with $n$ vertices, there is a matching tree $M(C_n)$ with two non-empty leaves corresponding to critical cells of size $k$ if $n=3k$, and one non-empty leaf corresponding to a critical cell of size $k$ if $n=3k\pm 1$.
\end{lemma}

\begin{remark} Lemmas~\ref{pathlemma}~and~\ref{cyclelemma} were not originally proved using matching trees, but it is simple to convert the proofs referenced in \cite{KozlovBook} to the language of matching trees.
\end{remark}

\subsection{Case: $4\mid 2n+2$}

The first step of our matching tree construction proceeds as follows: from $\Sigma(\emptyset,\emptyset)$, use vertex $1$ as a splitting vertex, yielding directed edges to new nodes $\Sigma(\{1\},\{2,4,\ldots,2n+2,c_1\})$ and $\Sigma(\emptyset,\{1\})$.  For node $\Sigma(\{1\},\{2,4,\ldots,2n+2,c_1\})$, we have that vertex $c_3$ is a matching vertex with respect to vertex $3$, thus we can create a directed edge to a new node labeled $\Sigma(\{1,c_3\},\{2,3,4,6,\ldots,2n+2,c_2,c_4,c_{3+n+1}\})$.  For this node, vertex $n+4$ is a free vertex, and we can create a new directed edge to a node labeled $\emptyset$.

Our construction now proceeds by repeating this first branching step on the node $\Sigma(\emptyset,\{1,2,3,\ldots,i\})$, where $1\leq i \leq n-2$.  Given node $\Sigma(\emptyset,\{1,2,3,\ldots,i\})$, use vertex $i+1$ as a splitting vertex, yielding two new nodes 
\[\Sigma(\emptyset,\{1,2,3,\ldots,i,i+1\}) \]
and 
\[ \Sigma(\{i+1\},\{1,2,3,\ldots,i,i+2,i+4,i+6,\ldots,j,c_{i+1}\}),\] 
where $j$ is $2n+1$ or $2n+2$ depending on the parity of $i$.  For the latter node, we have that vertex $c_{i+3}$ is a matching vertex with respect to $i+3$, thus we can create a new edge to a node labeled \[\Sigma(\{i+1,c_{i+3}\},\{1,2,3,\ldots,i,i+2,i+4,i+6,\ldots,j,c_{i+1},i+3,c_{i+2},c_{i+4},c_{i+4+n}\}).\]  Finally, vertex $i+4+n$ is a free vertex for this node (because $n$ is odd), and thus we can create a new edge to a node labeled $\emptyset$.  

After repeating this for all $1\leq i \leq n-2$, our only remaining leaf node in $M(E_{2n+2})$ is $\Sigma(\emptyset,\{1,2,3,\ldots,n-1\})$.  We now proceed by repeating a different matching process where we assume that $0\leq r \leq n-2$ and that $M(E_{2n+2})$ has only one leaf node, labeled $\Sigma(\emptyset,\{1,2,3,\ldots,n-1+r\})$.  We use vertex $n+r$ as a splitting vertex, creating two new edges to nodes labeled 
\[\Sigma(\emptyset,\{1,2,3,\ldots,n+r\})\]
and
\[\Sigma\left(\{n+r\},\{1,2,3,\ldots,n-1+r,n+r+1,n+r+3,\ldots,j,c_{n+r}\}\right),\] 
where $j$ is either $2n+1$ or $2n+2$ depending on the parity of $n+r$.  The first node is handled when we consider the case $r+1$, while for the second node we use $c_{n+r+2}$ as a matching vertex with respect to vertex $n+r+2$, yielding a new edge to a node labeled  \[\Sigma\left(\left\{\begin{array}{l} n+r, \\ c_{n+r+2}\end{array}\right\},\left\{
\begin{array}{l}
1,2,3,\ldots,n-1+r, \\ 
n+r+1,n+r+3,\ldots,j,c_{n+r}, \\
n+r+2,c_{n+r+1},c_{n+r+3},c_{r+1}
\end{array}
\right\}\right),\] where $j$ is as before.  Using the matching vertex $c_{n+r+4}$ with respect to the vertex $n+r+4$, we obtain a new edge to a node labeled \[\Sigma\left(\left\{\begin{array}{l}n+r, \\ c_{n+r+2}, \\c_{n+r+4}\end{array}\right\},\left\{
\begin{array}{l}
1,2,3,\ldots,n+r, \\
n+r+1,n+r+3,\ldots,j,c_{n+r}, \\
n+r+2,c_{n+r+1},c_{n+r+3},c_{r+1}, \\
n+r+4,c_{n+r+5},c_{r+3}
\end{array}\right\}\right).\]  Finally, vertex $c_{r+2}$ is a free vertex for this node, and we obtain a new edge to a node labeled $\emptyset$.

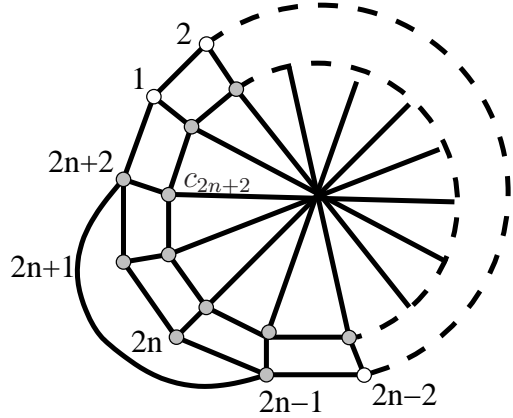
\begin{figure}[ht]
\begin{center}
\input{nOddPic.pstex_t}
\end{center}
\caption{We have suppressed some of the edges in the bipartite graph portion of $E_{2n+2}$ for clarity.}
\label{nOddPic}
\end{figure}
Following these iterations, our matching tree has only one non-empty leaf node, labeled $\Sigma(\emptyset,\{1,2,3,\ldots,2n-2\})$, represented as the graph in Figure~\ref{nOddPic}.  Our first step is to use vertex $2n$ as a splitting vertex, yielding two new edges to new nodes which we handle in subcases below.

\subsection{Subcase: $\Sigma(\{2n\},\{1,2,3,\ldots,2n-1,2n+1,c_{2n}\})$}\label{nOddPoint1}  Using vertex $c_{2n+2}$ as a splitting vertex makes vertex $2n+2$ a free vertex for 
\[\Sigma(\{2n\},\{1,2,3,\ldots,2n-1,2n+1,c_{2n},c_{2n+2}\}),\] 
leaving us to consider only the node 
\[\Sigma(\{2n,c_{2n+2}\},\{1,2,3,\ldots,2n-1,2n+1,2n+2,c_{2n},c_{2n+1},c_{1},c_{n+1}\}).\]  
Using vertex $c_{n-1}$ as a splitting vertex, we see that vertex $c_n$ is a free vertex for 
\[\Sigma(\{2n,c_{2n+2}\},\{1,2,3,\ldots,2n-1,2n+1,2n+2,c_{2n},c_{2n+1},c_{1},c_{n+1},c_{n-1}\}),\] 
leaving us to consider only the node 
\[\Sigma(\{2n,c_{2n+2},c_{n-1}\},\{1,2,3,\ldots,2n-1,2n+1,2n+2,c_{2n},c_{2n+1},c_{1},c_{n+1},c_n,c_{n-2}\}).\]  
The remaining vertices induce as a subgraph of $E_{2n+2}$ a copy of $EL_{n-4}$.  Thus, we are in a position to invoke Lemma~\ref{ELlemma}, and we see that the portion of the resulting matching tree $M(E_{2n+2})$ rooted from our remaining node yields a critical cell of size $2k+2$ if $n=4k+1$ and no critical cells if $n=4k+3$.

\subsection{Subcase: $\Sigma(\emptyset,\{1,2,3,\ldots,2n-2,2n\})$} Using vertex $c_{2n+2}$ as a splitting vertex makes vertex $2n+1$ a free vertex for \[\Sigma(\{c_{2n+2}\},\{1,2,3,\ldots,2n-2,2n,2n+2,c_1,c_{2n+1},c_{n+1}\}),\] leaving us to consider only the node \[\Sigma(\emptyset,\{1,2,3,\ldots,2n-2,2n,c_{2n+2}\}).\]  We now use vertex $2n+1$ as a splitting vertex, yielding two new edges to new nodes which we consider in two subsubcases.

\subsection{Subsubcase: $\Sigma(\emptyset,\{1,2,3,\ldots,2n-2,2n,c_{2n+2},2n+1\})$}\label{nOddPoint2}  By using vertex $2n-1$ as a matching vertex with respect to vertex $2n+2$, we need only consider node \[\Sigma(\{2n-1\},\{1,2,3,\ldots,2n-2,2n,c_{2n+2},2n+1,2n+2,c_{2n-1}\}).\]  Using vertex $c_{n-1}$ as a splitting vertex, we see that vertex $c_{2n+1}$ is a free vertex for \[\Sigma(\{2n-1,c_{n-1}\},\{1,2,3,\ldots,2n-2,2n,c_{2n+2},2n+1,2n+2,c_{2n-1},c_{2n},c_{n-2},c_n\}),\] hence we need only consider node \[\Sigma(\{2n-1\},\{1,2,3,\ldots,2n-2,2n,c_{2n+2},2n+1,2n+2,c_{2n-1},c_{n-1}\}).\]  Using vertex $c_{2n+1}$ as a splitting vertex, we see that vertex $c_{2n}$ is a free vertex for \[\Sigma(\{2n-1\},\{1,2,3,\ldots,2n-2,2n,c_{2n+2},2n+1,2n+2,c_{2n-1},c_{n-1},c_{2n+1}\}),\] hence we need only consider node \[\Sigma(\{2n-1,c_{2n+1}\},\{1,2,3,\ldots,2n-2,2n,c_{2n+2},2n+1,2n+2,c_{2n-1},c_{n-1},c_{2n},c_n\}).\]  Using vertex $c_{n+2}$ as a splitting vertex, we see that vertex $c_{n+1}$ is a free vertex for 
\[\Sigma \left(\left\{2n-1,c_{2n+1}\right\},\left\{
\begin{array}{l}1,2,3,\ldots,2n-2,2n,c_{2n+2},2n+1, \\
2n+2,c_{2n-1},c_{n-1},c_{2n},c_n,c_{n+2}
\end{array}\right\} \right),\]
 hence we only need to consider node 
 \[\Sigma\left(\left\{2n-1,c_{2n+1},c_{n+2}\right\},\left\{
 \begin{array}{l}
 1,2,3,\ldots,2n-2,2n,c_{2n+2},2n+1, \\
 2n+2,c_{2n-1}, c_{n-1},c_{2n},c_n,c_{n+1},c_1,c_{n+3}
 \end{array}
 \right\}\right).\]
The remaining vertices induce as a subgraph of $E_{2n+2}$ a copy of $EL_{n-5}$.  Thus, we are in a position to invoke Lemma~\ref{ELlemma}, and we see that the portion of the resulting matching tree $M(E_{2n+2})$ rooted from our remaining node yields a critical cell of size $2k+2$ if $n=4k+1$ and no critical cells if $n=4k+3$.

\subsection{Subsubcase: $\Sigma(\{2n+1\},\{1,2,3,\ldots,2n-2,2n,c_{2n+2},c_{2n+1},2n+2\})$}\label{nOddPoint3} We begin by using $c_{2n-1}$ as a matching vertex with respect to vertex $2n-1$, leaving us to consider only the node
\[\Sigma\left(\left\{2n+1,c_{2n-1}\right\},\left\{
\begin{array}{ll}
1,2,3,\ldots,2n-2,2n,c_{2n+2},c_{2n+1}, \\
2n+2,2n-1,c_{2n},c_{2n-2},c_{n-2}
\end{array}
\right\}\right).\]
We next use vertex $c_n$ as a matching vertex with respect to vertex $c_{n-1}$, leaving us to consider only the node
\[\Sigma\left(\left\{2n+1,c_{2n-1},c_n\right\},\left\{
\begin{array}{ll}
1,2,3,\ldots,2n-2,2n,c_{2n+2},c_{2n+1}, \\
2n+2,2n-1,c_{2n},c_{2n-2}, \\
c_{n-2},c_{n-1},c_{n+1}
\end{array}
\right\}\right).\]
We now use vertex $c_{n+2}$ as a splitting vertex, yielding the new nodes
\[\Sigma\left(\left\{2n+1,c_{2n-1},c_n\right\},\left\{
\begin{array}{ll}
1,2,3,\ldots,2n-2,2n,c_{2n+2},c_{2n+1}, \\
2n+2,2n-1,c_{2n},c_{2n-2},c_{n-2}, \\
c_{n-1},c_{n+1},c_{n+2}
\end{array}
\right\}\right)\]
and
\[\Sigma\left(\left\{2n+1,c_{2n-1},c_n,c_{n+2}\right\},\left\{
\begin{array}{ll}
1,2,3,\ldots,2n-2,2n,c_{2n+2},c_{2n+1}, \\
2n+2,2n-1,c_{2n},c_{2n-2},c_{n-2}, \\
c_{n-1},c_{n+1},c_1,c_{n+3}
\end{array}
\right\}\right).\]
The remaining vertices for the first of these nodes induce as a subgraph of $E_{2n+2}$ a copy of $EL_{n-5}$ while for the second of these nodes the remaining vertices induce as a subgraph of $E_{2n+2}$ a copy of $EL_{n-6}$.  Thus, we are in a position to invoke Lemma~\ref{ELlemma}, and we conclude that the portions of the resulting matching tree $M(E_{2n+2})$ rooted from these final two nodes yield a critical cell of size $2k+2$ if $n=4k+1$ and a critical cell of size $2k+3$ if $n=4k+3$.

\subsection{Summary of Case: $4\mid 2n+2$} When $n=4k+1$, there are $3$ critical cells, all of size $2k+2$.  When $n=4k+3$, there is a single critical cell of size $2k+3$.  As each of these cases yield cells of the same dimension, our resulting cell complex is a wedge of spheres, as desired.

\subsection{Case: $4\nmid 2n+2$}

We begin with the node $\Sigma(\emptyset,\emptyset)$ in $M(E_{2n+2})$ and use $1$ as a splitting vertex to create new nodes $\Sigma(\emptyset,\{1\})$ and $\Sigma(\{1\},\{2,4,\ldots,2n+2,c_1\})$.  On the latter node, we use vertex $c_3$ as a matching vertex with respect to the vertex $3$, creating a new edge to a new node $\Sigma(\{1,c_3\},\{2,4,\ldots,2n+2,c_1,3,n+4,c_5\})$.  Vertex $5$ is a free vertex for this node, yielding a final new edge to a new node labeled $\emptyset$.  The only remaining non-empty leaf node in our matching tree is $\Sigma(\emptyset,\{1\})$.

We inductively repeat this matching process on the node $\Sigma(\emptyset,\{1,2,\ldots,i\})$, where $1\leq i \leq 2n-3$.  We consider two cases.  If $i$ is odd, then using $i+1$ as a splitting vertex, we obtain new edges to new nodes labeled 
\[\Sigma(\emptyset,\{1,2,\ldots,i,i+1\})\] and 
\[\Sigma(\{i+1\},\{1,2,\ldots,i,i+2,i+4,\ldots,2n+1,c_{i+n+2}\}).\]  
For the latter node, we now use vertex $c_{i+n+4}$ as a matching vertex with respect to vertex $i+3$, yielding a new edge to a new node labeled \[\Sigma(\{i+1,c_{i+n+4}\},\{1,2,\ldots,i,i+2,i+4,\ldots,2n+1,i+n+4,c_{i+n+2},i+3,c_{i+n+6}\}).\]
Vertex $i+5$ is a free vertex for this node, yielding a new edge to a new node labeled $\emptyset$.  

If $i$ is even, then again using $i+1$ as a splitting vertex, we obtain new edges to new nodes labeled 
\[\Sigma(\emptyset,\{1,2,\ldots,i,i+1\})\]
and 
\[\Sigma(\{i+1\},\{1,2,\ldots,i,i+2,i+4,\ldots,2n+2,c_{i+1}\}).\]
For the latter node, we now use vertex $c_{i+3}$ as a matching vertex with respect to vertex $i+3$, yielding a new edge to a new node labeled \[\Sigma(\{i+1,c_{i+3}\},\{1,2,\ldots,i,i+2,i+4,\ldots,2n+2,c_{i+1},i+3,c_{i+5},n+i+4\}).\]
Vertex $i+5$ is a free vertex for this node, yielding a new edge to a new node labeled $\emptyset$.  Repeating this process for all $1\leq i\leq 2n-3$ results in a matching tree where all leaf nodes are labeled by $\emptyset$ except $\Sigma(\emptyset,\{1,2,3,\ldots,2n-2\})$, represented as the graph in Figure~\ref{nEvenPic}.

\begin{figure}[ht]
\begin{center}
\input{nEvenPic.pstex_t}
\end{center}
\caption{We have suppressed some of the edges in the bipartite graph portion of $E_{2n+2}$ for clarity.}
\label{nEvenPic}
\end{figure}
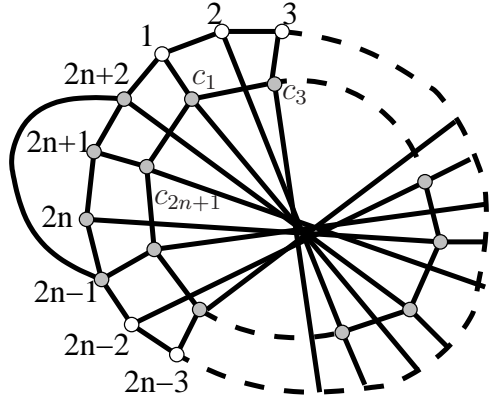

Using vertex $c_{n+1}$ as a splitting vertex for $\Sigma(\emptyset,\{1,2,3,\ldots,2n-2\})$ leaves us with two subcases.  We must address the cases $n=6k$ and $n=6k+2$ separately from $n=6k+4$ for the first part of the splitting.

\subsection{Subcase: $\Sigma(\emptyset,\{1,2,3,\ldots,2n-2,c_{n+1}\})$, $n=6k,6k+2$}  We use vertex $2n+1$ as a splitting vertex.  First, observe that the remaining vertices for node \[\Sigma(\{2n+1\},\{1,2,3,\ldots,2n-2,c_{n+1},2n,2n+2,c_{2n+1}\})\] induce as a subgraph of $E_{2n+2}$ a disjoint union of two paths of length $\frac{n}{2}$.  Hence, applying Lemma~\ref{pathlemma}, for $n=6k$ we have that the portion of the resulting matching tree $M(E_{2n+2})$ rooted at this node yields a single critical cell of size $2k+1$.  For $n=6k+2$, it yields no critical cells.  

Regarding node $\Sigma(\emptyset,\{1,2,3,\ldots,2n-2,c_{n+1},2n+1\})$, we see that vertex $2n-1$ is a matching vertex with respect to vertex $2n+2$, hence we only need consider the node \[\Sigma(\{2n-1\},\{1,2,3,\ldots,2n-2,c_{n+1},2n+1,2n+2,2n,c_{2n-1}\}).\] The remaining vertices for this node induce as a subgraph of $E_{2n+2}$ a disjoint union of a path of length $\frac{n}{2}+1$ and a path of length $\frac{n}{2}-2$.  Thus, applying Lemma~\ref{pathlemma}, for $n=6k$ we have that the portion of the resulting matching tree $M(E_{2n+2})$ rooted at this node yields no single critical cells while for $n=6k+2$ it yields a single critical cell of size $2k+2$.

\subsection{Subcase: $\Sigma(\emptyset,\{1,2,3,\ldots,2n-2,c_{n+1}\})$, $n=6k+4$} We first use vertex $c_{n+5}$ as a matching vertex, then proceed inductively to use vertex $c_{n+5+6l}$ as a matching vertex until $n+5+6l=2n-5$.  Note that this is always possible, since $n=6k+4$ and hence the path given by $c_{n+3}<c_{n+5}<\cdots<c_{2n-3}$ is of length $3k$.  This leaves us to consider the node 
\[\Sigma(\{c_{n+5},\ldots,c_{2n-5}\},\{1,2,3,\ldots,2n-2,c_{n+1},c_{n+3},c_{n+7},\ldots,c_{2n-3}\}).\]
We use $2n$ as a splitting vertex.  For the node 
\[\Sigma\left(\{c_{n+5},\ldots,c_{2n-5}\},\left\{\begin{array}{c}1,2,3,\ldots,2n-2,c_{n+1},\\ c_{n+3},c_{n+7},\ldots,c_{2n-3},2n\end{array}\right\}\right),\]
we use $c_{n-3}$ as a matching vertex and then inductively use vertex $c_{n-3-6l}$ as a matching vertex until we match on $c_1$.  The remaining vertices for the resulting node are $c_{2n-1},2n-1,2n+2,2n+1$, which induce a path of length four.  This reduces to a leaf node labeled $\emptyset$.

The remaining node to consider is 
\[\Sigma\left(\{c_{n+5},\ldots,c_{2n-5},2n\},\left\{\begin{array}{c}1,2,3,\ldots,2n-2,c_{n+1},\\ c_{n+3},c_{n+7},\ldots,c_{2n-3},2n-1,c_{n-1},2n+1\end{array}\right\}\right).\]
It is a straightforward observation that $2n+2$ is a free vertex for this node.  Thus, this portion of the matching tree yields no critical cells.

\subsection{Subcase: $\Sigma(\{c_{n+1}\},\{1,2,3,\ldots,2n-2,2n+2,c_{n-1},c_{n+3}\})$}  We use vertex $2n$ as a splitting vertex.  First, observe that the remaining vertices for node \[\Sigma(\{c_{n+1},2n\},\{1,2,3,\ldots,2n-2,2n+2,c_{n-1},c_{n+3},2n+1,2n-1\})\] induce a path of length $n-2$.  Thus, applying Lemma~\ref{pathlemma}, for $n=6k$ we have that the portion of the resulting matching tree $M(E_{2n+2})$ rooted at this node has no non-empty leaves.  For $n=6k+2$, this portion of the tree yields a single critical cell of size $2k+2$.  For $n=6k+4$, this portion of the tree yields a single critical cell of size $2k+3$.

For node \[\Sigma(\{c_{n+1}\},\{1,2,3,\ldots,2n-2,2n+2,c_{n-1},c_{n+3},2n\}),\] we have that $c_{2n+1}$ is a matching vertex with respect to vertex $2n+1$, yielding a new edge to a new node labeled \[\Sigma(\{c_{n+1},c_{2n+1}\},\{1,2,3,\ldots,2n-2,2n+2,c_{n-1},c_{n+3},2n,c_1,c_{2n-1},2n+1\}).\]  Finally, vertex $2n-1$ is a free vertex for this node, yielding a new edge to a new node labeled $\emptyset$.

\subsection{Summary of Case: $4 \nmid 2n+2$}  Thus, for $n=6k$ we can produce a matching tree with one critical cell of size $2k+1$.  For $n=6k+2$, our matching tree has two critical cells of size $2k+2$, and for $n=6k+4$ our matching tree has one critical cell of size $2k+3$.  As each of these cases yield cells of the same dimension, our resulting cell complex is a wedge of spheres, as desired.

% Bibliography

\bibliographystyle{plain}
\bibliography{Braun}

\end{document}

%% file: ELPic.pstex_t
\begin{picture}(0,0)%
\includegraphics{ELPic.pstex}%
\end{picture}%
\setlength{\unitlength}{4144sp}%
\begingroup\makeatletter\ifx\SetFigFontNFSS\undefined%
\gdef\SetFigFontNFSS#1#2#3#4#5{%
  \reset@font\fontsize{#1}{#2pt}%
  \fontfamily{#3}\fontseries{#4}\fontshape{#5}%
  \selectfont}%
\fi\endgroup%
\begin{picture}(4026,649)(-527,267)
\end{picture}%

%% file: ELPicProof.pstex_t
\begin{picture}(0,0)%
\includegraphics{ELPicProof.pstex}%
\end{picture}%
\setlength{\unitlength}{4144sp}%
\begingroup\makeatletter\ifx\SetFigFontNFSS\undefined%
\gdef\SetFigFontNFSS#1#2#3#4#5{%
  \reset@font\fontsize{#1}{#2pt}%
  \fontfamily{#3}\fontseries{#4}\fontshape{#5}%
  \selectfont}%
\fi\endgroup%
\begin{picture}(4019,3493)(-546,-2602)
\end{picture}%

%% file: nOddPic.pstex_t
\begin{picture}(0,0)%
\includegraphics{nOddPic.pstex}%
\end{picture}%
\setlength{\unitlength}{4144sp}%
\begingroup\makeatletter\ifx\SetFigFontNFSS\undefined%
\gdef\SetFigFontNFSS#1#2#3#4#5{%
  \reset@font\fontsize{#1}{#2pt}%
  \fontfamily{#3}\fontseries{#4}\fontshape{#5}%
  \selectfont}%
\fi\endgroup%
\begin{picture}(3015,2476)(-329,-1516)
\put(721,-151){\makebox(0,0)[lb]{\smash{{\SetFigFontNFSS{12}{14.4}{\rmdefault}{\mddefault}{\updefault}$c_{2n+2}$}}}}
\end{picture}%

%% file: nEvenPic.pstex_t
\begin{picture}(0,0)%
\includegraphics{nEvenPic.pstex}%
\end{picture}%
\setlength{\unitlength}{4144sp}%
\begingroup\makeatletter\ifx\SetFigFontNFSS\undefined%
\gdef\SetFigFontNFSS#1#2#3#4#5{%
  \reset@font\fontsize{#1}{#2pt}%
  \fontfamily{#3}\fontseries{#4}\fontshape{#5}%
  \selectfont}%
\fi\endgroup%
\begin{picture}(2906,2403)(-570,-1084)
\put(541,794){\makebox(0,0)[lb]{\smash{{\SetFigFontNFSS{12}{14.4}{\rmdefault}{\mddefault}{\updefault}$c_1$}}}}
\put(1081,704){\makebox(0,0)[lb]{\smash{{\SetFigFontNFSS{12}{14.4}{\rmdefault}{\mddefault}{\updefault}$c_3$}}}}
\put(316, 74){\makebox(0,0)[lb]{\smash{{\SetFigFontNFSS{12}{14.4}{\rmdefault}{\mddefault}{\updefault}$c_{2n+1}$}}}}
\end{picture}%

%% file: Braun_IndependenceComplexesStableKneserGraphs_Arxiv_Dec2009.bbl
\begin{thebibliography}{10}

\bibitem{BabsonKozlovLovaszConjecture}
Eric Babson and Dmitry~N. Kozlov.
\newblock Proof of the {L}ov\'asz conjecture.
\newblock {\em Ann. of Math. (2)}, 165(3):965--1007, 2007.

\bibitem{BjornerDeLongueville}
Anders Bj\"{o}rner and Mark de~Longueville.
\newblock Neighborhood complexes of stable kneser graphs.
\newblock {\em Combinatorica}, 23(1):23--34, 2003.

\bibitem{LinussonGridGraphs}
Mireille Bousquet-M{\'e}lou, Svante Linusson, and Eran Nevo.
\newblock On the independence complex of square grids.
\newblock {\em J. Algebraic Combin.}, 27(4):423--450, 2008.

\bibitem{BraunSymmetry}
Benjamin Braun.
\newblock Symmetries of the stable kneser graphs.
\newblock Submitted, available at www.ms.uky.edu/$\sim$braun.

\bibitem{EhrenborgHetyei}
Richard Ehrenborg and G{\'a}bor Hetyei.
\newblock The topology of the independence complex.
\newblock {\em European J. Combin.}, 27(6):906--923, 2006.

\bibitem{EngstromClawFree}
Alexander Engstr{\"o}m.
\newblock Independence complexes of claw-free graphs.
\newblock {\em European J. Combin.}, 29(1):234--241, 2008.

\bibitem{EngstromWitten}
Alexander Engstr{\"o}m.
\newblock Upper bounds on the {W}itten index for supersymmetric lattice models
  by discrete {M}orse theory.
\newblock {\em European J. Combin.}, 30(2):429--438, 2009.

\bibitem{FormanMorseTheory}
Robin Forman.
\newblock Morse theory for cell complexes.
\newblock {\em Adv. Math.}, 134(1):90--145, 1998.

\bibitem{JonssonHardSquares}
Jakob Jonsson.
\newblock Hard squares with negative activity and rhombus tilings of the plane.
\newblock {\em Electron. J. Combin.}, 13(1):Research Paper 67, 46 pp.
  (electronic), 2006.

\bibitem{JonssonBook}
Jakob Jonsson.
\newblock {\em Simplicial complexes of graphs}, volume 1928 of {\em Lecture
  Notes in Mathematics}.
\newblock Springer-Verlag, Berlin, 2008.

\bibitem{KozlovBook}
Dmitry Kozlov.
\newblock {\em Combinatorial algebraic topology}, volume~21 of {\em Algorithms
  and Computation in Mathematics}.
\newblock Springer, Berlin, 2008.

\bibitem{KozlovIndCycles}
Dmitry~N. Kozlov.
\newblock Complexes of directed trees.
\newblock {\em J. Combin. Theory Ser. A}, 88(1):112--122, 1999.

\bibitem{LovaszChromaticNumberHomotopy}
L.~Lov{\'a}sz.
\newblock Kneser's conjecture, chromatic number, and homotopy.
\newblock {\em J. Combin. Theory Ser. A}, 25(3):319--324, 1978.

\bibitem{MatousekBorsukUlam}
Ji{\v{r}}{\'{\i}} Matou{\v{s}}ek.
\newblock {\em Using the {B}orsuk-{U}lam theorem}.
\newblock Universitext. Springer-Verlag, Berlin, 2003.
\newblock Lectures on topological methods in combinatorics and geometry,
  Written in cooperation with Anders Bj\"orner and G\"unter M. Ziegler.

\bibitem{SchrijverGraphs}
A.~Schrijver.
\newblock Vertex-critical subgraphs of {K}neser graphs.
\newblock {\em Nieuw Arch. Wisk. (3)}, 26(3):454--461, 1978.

\bibitem{Talbot}
John Talbot.
\newblock Intersecting families of separated sets.
\newblock {\em J. London Math. Soc. (2)}, 68(1):37--51, 2003.

\bibitem{Thapper}
Johan Thapper.
\newblock Independence complexes of cylinders constructed from square and
  hexagonal grid graphs, 2008.
\newblock http://www.citebase.org/abstract?id=oai:arXiv.org:0812.1165.

\end{thebibliography}
